\documentclass[11pt, reqno]{amsart}
\usepackage{times,amsmath,amssymb,amsthm,amscd, color}
\usepackage[all]{xy}
\usepackage{cite}
\usepackage{enumerate}
\usepackage{ulem}

\numberwithin{equation}{section}

\newtheorem{thm}{Theorem}[section]
\newtheorem{cor}[thm]{Corollary}
\newtheorem{lem}[thm]{Lemma}

\newcommand{\CC}{{\mathbb C}}

\newcommand{\NN}{{\mathbb N}}

\def\BB{ \mathbb{B}}
\def\SS{ \mathbb{S}}

\def\om{\omega}

\def\msk{\medskip}

\begin{document}
	\title[Difference of weighted composition operators]{Difference of weighted composition operators on weighted Bergman spaces over the unit Ball }
	\author{ Lian Hu,  Songxiao Li$^\ast$ and Yecheng Shi}
	\address{Lian Hu\\ Institute of Fundamental and Frontier Sciences, University of Electronic Science and Technology of China,
		610054, Chengdu, Sichuan, P.R. China.}	\email{hl152808@163.com  }

	\address{Songxiao Li\\ Department of Mathematics, Shantou University, Shantou Guangdong 515063, P.R. China. }
	\email{jyulsx@163.com}
	
\address{Yecheng Shi\\    School of Mathematics and Statistics, Lingnan Normal University,  Zhanjiang 524048, Guangdong, P. R. China}  \email{  09ycshi@sina.cn}

	\subjclass[2010]{32A36, 47B38}
	\begin{abstract}   In this paper, we characterize the boundedness and compactness  of differences of weighted composition operators from weighted Bergman spaces $A^p_\om$ induced by a doubling weight $\om$ to Lebesgue spaces $L^q_\mu$ on the unit ball for full $0<p,q<\infty$, which extend many results on the unit disk. As a byproduct, a new characterization of $q$-Carleson the measure for $A^p_\om$ in terms of the Bergman metric ball is also presented.
		\thanks{$\ast$ Corresponding author.}
		\vskip 3mm \noindent{\it Keywords}:Difference, Weighted Bergman space, weighted composition operator, doubling weight, Carleson measure.
		\end{abstract}	\maketitle
	
	\section{Introduction}
	Denote by $H(\BB_n)$ and $\mathcal{S}(\BB_n)$ the space of analytic functions and holomorphic self-maps on the unit ball $\BB_n$ of the complex $n$-space $\CC^n$, respectively.
Given $\varphi \in \mathcal{S}(\BB_n)$ and a Borel function $u$ on $\BB_n$, the weighted composition operator $uC_\varphi$ with symbol $\varphi$ and weight $u$
is defined by 
$$
uC_\varphi f= f\circ \varphi
$$
for $f\in H(\BB_n)$. 
When $ u\equiv 1$, the operator $uC_\varphi$ becomes the classical composition operator $C_\varphi$.
The  theory of (weighted) composition operators on various settings has been extensively studied in recent decades
and  we refer to the books \cite{cm} and \cite{s}.  
The difference of composition operators has become an attractive topic in various areas of function theory and operator theory.
Recently,
Choe et al. [\citen{ccky,ccky2}] completely characterized the bounded and compact differences of weighted composition operators from   standard weighted Bergman spaces to Lebesgue spaces on the unit disk.
Subsequently, Choe et al. \cite{cckp} extend their results to the unit ball.  
Chen \cite{c} generalized their results to the doubling weight setting and characterized the difference of composition operators from weighted Bergman spaces $A^p_\om$ to Lebesgue spaces $L^q_\mu$ for full $0<p,q<\infty$.
The main purpose of  the current paper is to generalize the main results in \cite{c} to the unit ball setting.

To begin with, we recall some definitions. 	A positive, measurable and  integrable  function $\omega$, satisfying $\om(z)=\om(|z|)$ for all $z\in\BB_n$, is called a radial weight on $\BB_n$. Put  $\hat{\omega}(r)=\int_r^1\omega(s)ds$ for $r\in [0,1)$. 
A radial weight  $\omega$ is called a doubling weight, denoted by $\omega\in \hat{\mathcal{D}}$, if there is a constant $C=C(\om)\ge 1$ such that  $$
\hat{\omega}(r)<C\hat{\omega}\left( { r+1 \over 2 }\right)  {\rm ~for ~} r\in [0,1).$$
A radial weight  $\omega$ is called a reverse doubling weight, denoted by $\om\in \breve{\mathcal{D}}$, if  there exists  constants $C=C(\om),K=K(\om)>1$ such that
$$
\hat{\omega}(r)\ge C\hat{\omega}\left(1-{1-r\over K}\right) {\rm ~for ~} r\in [0,1).
$$
Define $\mathcal{D}=\hat{\mathcal{D}}\cap\breve{\mathcal{D}}$, called the two-sides doubling weight.
	See \cite{p,pr} and the references therein for more results on  doubling weights.

For $0<p<\infty$ and a weight function $\omega$, the weighted Bergman space $A^p_\omega$  is the space of all functions $f\in H(\BB_n)$ such that 
$$
\|f\|^p_{A^p_\omega}= \int_{\BB_n}|f(z)|^p\omega(z) dV (z) <\infty,
$$
where $dV$ is the normalized volume measure on $\BB_n$. Clearly, $A^p_\om$ is a Banach space when $1\le p<\infty$ and a complete metric space with the distance $\rho(f,g)=\|f-g\|^p_{A^p_\om}$ when $0<p<1$. 
Taking $\om(z)={\Gamma(n+\alpha+1)  \over \Gamma(n+1)\Gamma(\alpha+1) }(1-|z|^2)^\alpha,z\in \BB_n$ with $\alpha>-1$,  $A^p_\om$ becomes the standard weighted Bergman space $A^p_\alpha$. 
 
 To present our results,  we introduce several notations.
 For $\varphi,\psi \in \mathcal{S}(\BB_n)$, put
 $$
 \rho(z):=\rho(\varphi(z),\psi(z)),\,\,z\in\BB_n,
 $$
where $\rho(\cdot,\cdot)$ denotes the pseudohyperbolic distance on $\BB_n$; see Sect. 2.2.
Given a positive Borel measure $\nu$ on $\BB_n$, the pullback measure $\nu\circ\varphi^{-1}$ is defined by $$ (\nu\circ\varphi^{-1})(E)= \nu[\varphi^{-1}(E)]$$
 for Borel sets $E\subset \BB_n$.
Let $0<r<1$, $0<\alpha,q<\infty$, $u,v$ be weights and $\mu$ be  a positive Borel measure  on $\BB_n$.
The pullback measures $ \eta$ and $ \sigma_r$ 
 are defined by
$$
\eta=(|\rho u|^q d\mu)\circ\varphi^{-1} + (|\rho v|^q d\mu)\circ\psi^{-1},
$$
and
$$
\sigma_r = (\chi_{G_r}|u-v|^qd\mu)\varphi^{-1} + (\chi_{G_r}|u-v|^qd\mu)\psi^{-1},
$$
respectively, where $\chi_{G_r}$ is the characteristic function of the set
$$
G_r:= \{z\in \BB_n:\rho(z)<r\}.
$$

Recall that for a Banach space $X$ of analytic functions  and a positive Borel measure $\mu$ on $\BB_n$,   $\mu$ is called a (varnishing) $q$-Carleson measure for $X$ if the identity operator $I_d:X\to L^q_\mu$ is bounded (compact).
In this paper, we completely characterized the boundededness and compactness of differences of weighted composition operators from   weighted Bergman spaces $ A^p_\om$ to Lebesgue spaces $ L^q_\mu$ for full $0<p,q<\infty$ on the unit ball. 
Our main results are divided into two cases: $p\le q$ and $  p>q$, which are stated as follows.

\begin{thm}\label{pqbddcom}
  	Let $0<p\le q<\infty$, $0<r<1$,  $\om\in \mathcal{D}$, $\varphi,\psi \in \mathcal{S}(\BB_n)$, $u,v$ be weights 
	and $\mu$ be a positive Borel measure on $\BB_n$. 
  Then 
   $ uC_\varphi-vC_\psi: A^p_\om \to L^q_\mu$ is bounded (resp. compact) if and only if
  $\eta+ \sigma_r$ is a (resp. varnishing) $q$-Carleson measure for $A^p_\om$.
\end{thm}

\begin{thm} \label{qpcom}
	Let $0<q<p<\infty$, $0<r<1$, $\om\in \mathcal{D}$, $\varphi,\psi \in \mathcal{S}(\BB_n)$, $u,v$ be weights  
	and 
	$\mu$ be a positive Borel measure on $\BB_n$. 
	Then the following statements are equivalent:
	\begin{enumerate}
		\item[(i)] $ uC_\varphi-vC_\psi: A^p_\om \to L^q_\mu$ is bounded;
		\item[(ii)] $ uC_\varphi-vC_\psi: A^p_\om \to L^q_\mu$ is compact;
		\item[(iii)] $\eta+ \sigma_r$ is a $q$-Carleson measure for $A^p_\om$.
	\end{enumerate}
\end{thm}

This paper is organized as follows. 
 We give some preliminaries in Section 2. In Section 3, some new characterizations of the (varnishing) $q$-Carleson measure for $A^p_\om$ for full $0<p,q<\infty$ are given. In Section 4, we prove our main results.

Throughout this paper, let $ C$ be a positive constant that may change from one step to the next. For nonnegative quantities $A$ and $B$, we say that $A\lesssim B$ if there is a constant $C>0$ such that $A\leq CB$. The symbol $A\approx B$ means that $A\lesssim B\lesssim A$.
\msk

	\section{Preliminaries}	
In this section, we collect some basic facts that are important to our proof.

\subsection{Compact operator}
Let $X,Y$ be topological vector spaces.  A linear map $T:X\to Y$ is said to be compact if the image of any bounded sequence in $X$ has a subsequence that converges in $Y$.
To study the compactness of $T:A^p_\om\to L^q_\mu$, we need the following lemma. Its verification is a simplified modification of the proof for Proposition 3.11 in \cite{cm}.

\begin{lem}\label{cpt}
	Let $0<p,q<\infty$, $\om\in \mathcal{D}$ and $\mu$ be a positive Borel measure on $\BB_n$. Suppose that $ T:A^p_\om\to L^q_\mu$ is linear and bounded. Then $ T:A^p_\om\to L^q_\mu$ is compact if and only if $ Tf_k\to 0$ in $L^q_\mu$, where $\{f_k\}$ is bounded in $A^p_\om$  and converges to $0$ uniformly on any compact subset of $\BB_n$.
\end{lem}

\subsection{Pseudohyperbolic distance}

Recall that the pseudohyperbolic distance between $a, z\in \BB_n$ is defined by
 $$
 \rho(a,z)=|\varphi_a(z) |,
 $$ where $\varphi_a(z) $ is the M\"obius transformation in $\BB_n$.
Let $\beta(\cdot,\cdot)$ denote the Bergman metric on $\BB_n$, that is,
$$
\beta(z,w)={1\over 2}\log{1+\rho(z,w)  \over  1-\rho(z,w)},\,\, z,w\in\BB_n,
$$ 
and $D(z,r)=\{w\in \BB_n:\beta(z,w)<r\}$ for $r\in (0,1)$ be the Bergman metric ball centered at $z$ with radius $r$.
Let $P_z$ be the orthogonal projection from $\CC^n$ onto the one dimensional subspace $[z]$ generated by $z$ and $Q_z$ be the orthogonal projection from $\CC^n$ onto $\CC^n\ominus [z]$.
For any $z\in \BB_n\backslash \{0\}$ and $r>0$, $D(z,r)$ is an ellipsoid consisting of all points $w\in \BB_n$ that satisfy
\begin{equation}\label{DPQ}
{ |P_z(w)-c|^2 \over  s^2R^2} + { |Q_z(w)|^2 \over sR^2} <1,
\end{equation}
where
$$
R=\tanh(r),\,c= { (1-R^2)z \over 1-R^2|z|^2 },\,s= { 1-|z|^2 \over 1-R^2|z|^2 }.
$$
It is clear that
\begin{equation}\label{wDzr}
1-|z|\approx |1-|w|\approx |1-\langle z,w\rangle|
\end{equation}
and 
$$
|1-\langle z,a\rangle|\approx |1-\langle a,w\rangle|
$$
for any $a\in \BB_n$ and $z,w\in \BB_n$ with $ \beta(z,w)<r$,  where the constants suppressed depend only on $r$ and $n$. 

For $\xi\in \partial\BB_n$ and $\delta>0$, let
$$
S(\xi,\delta):=\{z\in\BB_n|~|1-\langle z,\xi\rangle|<\delta\}
$$
be the Carleson tube at $\xi$. From (2.13) in \cite{cckp}, we see that 
 \begin{equation}\label{De1r}
	D(te_1,r)\subset S\left( e_1, {2(1-t) \over 1-\tanh^{-1} r}\right) 
\end{equation}
for all $r,t\in(0,1)$.

\subsection{Separated sequences and lattices}
A sequence of points $\{a_j\}\subset \BB_n$ is said to be separated if there is a  constant $ \delta>0$ such that the Bergman metric $\beta(a_i,a_j)\ge \delta$ for all $ i$ and $j$ with $i\ne j$. This implies that there exists $r>0$ such that the Bergman metric balls $D(a_j,r)$ are pairwise disjoint.
\begin{lem}\label{cover}\cite[Theorem 2.23]{z}
	There exists a positive integer $N$ such that, for any $0<r\le1$ we can find a sequence $\{a_k\}\in \BB_n$ with the following properties:
	\begin{enumerate}
		\item[(i)] $\BB_n=\cup_kD(a_k,r)$;
		\item[(ii)] The sets $D(a_k,{r\over 4})$ are mutually disjoint;
		\item[(iii)] Each point $z\in\BB_n$ belongs to at most $N$ of the sets $D(a_k,4r)$.
	\end{enumerate}	
\end{lem}

Any sequence $\{a_k\}$ satisfying the above conditions is called an $r$-lattice in the Bergman metric. Clearly, any $r$-lattice sequence is separated.

\subsection{Function property}

	From  [\citen{dl2}, Lemma 2.2], we know that the weight $\omega$ has the following properties.

\begin{lem}\label{lem1}
	Let $\omega\in \mathcal{D}$. For $a,z\in\BB_n$, then the following statements hold:
	\begin{enumerate}
		\item[(i)]$\omega(S_a)\approx (1-|a|)^n\int_{|a|}^1\omega(s)ds$, where $S_a$ is the Carleson block;
		\item[(ii)] $\hat{\omega}(a)\approx\hat{\omega}(z)$, if $1-|a|\approx 1-|z|$;
		\item[(iii)] There are $0<\alpha:=\alpha(\om)\le \beta:=\beta(\om)$ and $C=C(\om)>1$ such that
		$$
		{1\over C}\left({1-s \over 1-t} \right)^\lambda\le { \hat{\om}(s) \over \hat{\om}(t) }\le C \left({1-s \over 1-t} \right)^\beta,0\le s\le t<1.
		$$
	\end{enumerate}
\end{lem}
For $\omega\in \mathcal{D}$,   the twisted weight $W$  is defined as follows.
  $$W(z)=W_\om(r):={ \hat{\om}(z) \over 1-|z| },\,\,z\in \BB_n.$$

The following result may have appeared in some literature, but for the benefit of the reader, we provide a brief proof here.

\begin{lem}\label{ApwW}
	Let $0<p<\infty, \om\in \mathcal{D}$. Then $\|f\|_{A^p_\om}\approx \|f\|_{A^p_{W}}$ for all $f\in H(\BB_n)$.
\end{lem}

\begin{proof}
	Let $r_k=1-2^{-k}$ for all $k\in \NN \cup \{0\}$. 
	 By $\om\in \mathcal{D}$,
there is $C=C(\om)\ge 1$ such that
	\begin{align*}
	\widehat{W}(r_k)= &\int_{r_k}^{r_{k+1}} { \hat{\om}(s) \over 1-s }ds +\int_{r_{k+1}}^1  W(s)ds \\
	\le &\hat{\om}(r_k) \int_{r_k}^{r_{k+1}} { 1 \over 1-s }ds +\int_{r_{k+1}}^1  W(s)ds\\
	\le &C \log2 \cdot \hat{\om}(r_{k+1})  +\int_{r_{k+1}}^1  W(s)ds \le C_1 \widehat{W}(r_{k+1})
		\end{align*}
		for some $C_1=C_1(\om)\ge 1$.
	Then similar to the proof of Lemma 2.1(${\rm i}\Rightarrow {\rm iii}$) in \cite{p}, there are $\gamma:=\gamma(\om)>0$ and $C:=C(\om)>0$ such that 
	$$
	\int_0^t\left({1-t\over 1-s} \right) ^\gamma W(s)ds\le C\widehat{W}(t),\,\,0\le t<1,
	$$
	which gives that $W\in \hat{\mathcal{D}}$ by Lemma A(iii) in \cite{prs2}. Applying Theorem 1 in \cite{dlls2}, $\|f\|_{A^p_\om}\approx \|f\|_{A^p_{W}}$ for all $f\in H(\BB_n)$ if $W(S_a)\approx \om(S_a)$ for $a\in \BB_n$.
	Since $W$ and $\om$ are radial weights, we get the desired result if $\widehat{W}(r)\approx\hat{\om}(r) $ for $0\le r<1$.
	Using Lemma B(iii) in \cite{prs2}, for $0\le r<1$, there is a $C:=C(\om)>0$ such that
	$$
	\widehat{W}(r)=\int_r^1{ \hat{\om}(s) \over 1-s }ds\le C\hat{\om}(r)
	$$
	and by Lemma A(ii) in \cite{prs2}, there are constants $ \beta:=\beta(\om)>0$ and $C:=C(\om)>0$ such that
	$$
	\widehat{W}(r)=\int_r^1{ \hat{\om}(s) \over 1-s }ds
	\ge C\int_r^1 {\hat{\om}(r)  \over 1-s}\left( 1-s \over 1-r \right) ^\beta ds\ge C \hat{\om}(r).
	$$
	Therefore, $\widehat{W}(r)\approx\hat{\om}(r) $ for $0\le r<1$. The proof is complete.
\end{proof}

 The following Lemma deals with the integrability of doubling weights.
 
 \begin{lem}\label{whatw}
 	If $ \omega\in{\mathcal{D}}$, then there is $\lambda_0=\lambda_0(\omega)\ge 0$ such that$$
 	\int_{\BB_n} {\omega(z)  \over |1-\langle z,a\rangle|^{\lambda n+n} }dV(z)\approx { \hat{\om}(a) \over  (1-|a|)^{\lambda n}},a\in \BB_n,$$for all $\lambda>\lambda_0$.
 \end{lem}
 
 \begin{proof}
 	Using Theorem 1.12 in \cite{z} and Lemma 2.1 (iii) in \cite{p}, we have
 	\begin{align*}
 		\int_{\BB_n} {\omega(z)  \over |1-\langle z,a\rangle|^{\lambda n+n} }dV(z)=&2n\int_0^1	\omega(r)r^{2n-1}\int_{\SS_n}{1\over |1-\langle r\xi,a\rangle|^{\lambda n+n}}d\sigma(\xi)dr\\
 		\approx& \int_0^1{ \omega(r)r^{2n-1} \over (1-|a|r)^{\lambda n} }dr\\
 		\le& \int_0^{|a|}{ \omega(r) \over (1-r)^{\lambda n} }dr+\int_{|a|}^1{ \omega(r) \over (1-|a|)^{\lambda n} }dr\\
 		\lesssim& { \int_{|a|}^1\omega(r)dr \over  (1-|a|)^{\lambda n}}={ \hat{\om}(a) \over  (1-|a|)^{\lambda n}}
 	\end{align*}
 	and
 	\begin{align*}
 		\int_{\BB_n} {\omega(z)  \over |1-\langle z,a\rangle|^{\lambda n+n} }dV(z)=&2n\int_0^1	\omega(r)r^{2n-1}\int_{\SS_n}{1\over |1-\langle r\xi,a\rangle|^{\lambda n+n}}d\sigma(\xi)dr\\
 		\gtrsim&\int_{|a|}^1{ \omega(r)r^{2n-1} \over (1-|a|)^{\lambda n} }dr\approx { \int_{|a|}^1\omega(r)dr \over  (1-|a|)^{\lambda n}}={ \hat{\om}(a) \over  (1-|a|)^{\lambda n}}.
 	\end{align*}
 	The proof is complete.
 \end{proof}

To study the boundedness and compactness of $uC_\varphi-vC_\psi:A^p_\om\to L^q_\mu$, we choose	$ K_a(z) = {1 \over 1-\langle a,z\rangle } ,z,a\in \BB_n$ as the test function in  the case $p\le q$.
From the above lemma, we see that 
 \begin{equation}\label{Kas}
\|K_a^s\|_{A^p_\om}^p \approx {  \hat{\om}(a) \over  (1-|a|)^ {ps-n } }
\end{equation}
for any $s>{ \lambda n+n \over  p}$.
In the case $p>q$, we need to consider another test function, given by the following Lemma \ref{Flambda}.
Before presenting  Lemma \ref{Flambda}, we first introduce a necessary lemma, whose proof is similar to that of Lemma 2, and we only  give the key steps.

\begin{lem}\label{whatwt}
	If $ \omega\in {\mathcal{D}}$, then there is $\beta=\beta(\om)>0$ and $\lambda_0=\lambda_0(\omega)\ge 0$ such that$$
	\int_{\BB_n} {\omega(z)  \over(1-|z|^2)^t |1-\langle z,a\rangle|^{\lambda n+n} }dV(z)\lesssim { \hat{\om}(a) \over  (1-|a|)^{\lambda n+t}},a\in \BB_n,$$for all $t>\beta,\lambda>\lambda_0$.
\end{lem}

\begin{proof}
	Let $\beta=\beta(\om)$ be the constant in Lemma \ref{lem1}(iii).
Applying Theorem 1.12 in \cite{z} and Lemma 2.1 (iii) in \cite{p}, we have
	\begin{align*}
		\int_{\BB_n} {\omega(z)  \over (1-|z|^2)^t|1-\langle z,a\rangle|^{\lambda n+n} }dV(z)
		\approx& \int_0^1{ \omega(r)r^{2n-1} \over (1-r)^t (1-|a|r)^{\lambda n} }dr\\
		\le& \int_0^{|a|}{ \omega(r) \over (1-r)^{\lambda n+t} }dr+\int_{|a|}^1{ \omega(r) \over (1-r)^{\lambda n+t} }dr.
	\end{align*}
Note that	 Lemma \ref{lem1}(iii) gives that
\begin{align*}	
	\int_{|a|}^1{ \omega(r) \over (1-r)^{\lambda n+t} }dr\le & \int_{|a|}^1 {\om(r) \over  (1-|a|)^{\lambda n+t} } { \hat{\om}(a) \over  \hat{\om}(r)} dr
	\le { \hat{\om}(a) \over  (1-|a|)^{\lambda n+t}},
\end{align*}
which deduces the desired result.
\end{proof}

The unit disk case of the following lemma is given by Pel\'aez and R\"atty\"a in [\citen{prs}, Theorem 1].  Here, we prove the high-dimensional case by using a different method from theirs.

\begin{lem}\label{Flambda}
Let $0<p<\infty, \om\in \mathcal{D}$ and $\{a_k\}$ be a separated sequence in $\BB_n$. If $ t>n+{ \beta(\om)+\lambda(\om)n+n \over  p}$ for some $ \beta(\om),\lambda(\om)>0$ and $\lambda=\{\lambda_k\}\in l^p$, then the function $F$ defined by 
$$
F(z):=\sum_{k=1}^\infty \lambda_k  {(1-|a_k|^2)^{t-{n\over p}}  \over  \hat{\om}(a_k)^{1\over p} (1-\langle z,a_k\rangle )^t}
$$
belongs to $ A^p_\om(\BB_n)$ and there is a constant $C:=C(t,p,\om)>0$ such that
$$
\|F\|_{A^p_\om}\le C\sum_{k=1}^n |\lambda_k|^p=C\|\lambda\|_{l^p}.
$$
\end{lem}

\begin{proof}
   Let 
   $$
   g_k(z):={ (1-|a_k|^2)^{t-{n\over p}}  \over  \hat{\om}(a_k)^{1\over p} (1-\langle z,a_k\rangle )^t },\,\,z\in \BB_n.
   $$
   Then $g_k\in A^p_\om$ with $\|g_k\|_{A^p_\om}\approx 1 $ by (\ref{Kas}) for $t>{n\over p}$. Therefore, for $0<p\le 1$, there exists $C:=C(p,\om)>0$ such that
   $$
   \|F\|_{A^p_\om}\le \sum_{k=1}^\infty|\lambda_k|^p\|g_k\|_{A^p_\om}^p\le C \sum_{k=1}^\infty|\lambda_k|^p.
   $$
   Next consider the case $1<p<\infty$. For $g\in H(\BB_n)$, let 
   $$
   \Lambda g(z):=\int_{\BB_n}{ (1-|w|^2)^{t-n-1} \over  |1-\langle z,w\rangle| ^t}g(w)dV(w),\,\,z\in \BB_n.
   $$
   Consider the function 
   $$
   f(z):=\sum_{k=1}^\infty { |\lambda_k| \chi_{D(a_k,r)}(z) \over \om(a_k)^{1\over p} V\left( D(a_k,r)\right) ^{1\over p} },
   $$
   where $\chi$ is the characteristic function. We get
   	\begin{align*}
   	\Lambda f(z)=&\sum_{k=1}^\infty { |\lambda_k|  \over \om(a_k)^{1\over p} V\left( D(a_k,r)\right) ^{1\over p} }\int_{D(a_k,r)}{ (1-|w|^2)^{t-n-1} \over  |1-\langle z,w\rangle| ^t}dV(w)\\
   	\approx& \sum_{k=1}^\infty { |\lambda_k|  \over  V\left( D(a_k,r)\right) ^{1\over p} } \cdot{  (1-|a_k|^2)^{t+{1\over p}-n-1}\over \hat{\om}(a_k)^{1\over p} |1-\langle z,a_k\rangle| ^t} \int_{D(a_k,r)} {  \hat{\om}(|w|)^{1\over p}\over \om(w)^{1\over p}(1-|w|)^{1\over p}} dV(w)\\
   	\gtrsim & \sum_{k=1}^\infty { |\lambda_k|  \over \hat{\om}(a_k)^{1\over p} V\left( D(a_k,r)\right) ^{1\over p} } \cdot{  (1-|a_k|^2)^{t+{1\over p}-n-1}\over  |1-\langle z,a_k\rangle| ^t} V(D(a_k,r)) \\
   		\approx& \sum_{k=1}^\infty |\lambda_k| { (1-|a_k|^2)^{t-{n\over p}} \over \hat{\om}(a_k)^{1\over p} |1-\langle z,a_k\rangle| ^t }\ge |F(z)|
   		\end{align*}
   		for all $z\in \BB_n$.  We claim that $ \Lambda$ is bounded on $L^p_\om$ at this moment. Then there is a constant $C:=C(t,p,\om)>0$ such that
   \begin{equation}\label{equ1}
   		\|F\|_{A^p_\om}\le C \|f\|_{L^p_\om}.
   	\end{equation}
   Noting that each point $z\in\BB_n$ belongs to at most $N$ of the sets $D(a_k,4r)$ by Lemma \ref{cover}, we have 
   $$
   |f(z)|^p\le N^{p-1} \sum_{k=1}^\infty { |\lambda_k|^p \chi_{D(a_k,4r)}(z) \over \om(a_k) V\left( D(a_k,r)\right) }, \, z\in \BB_n.
   $$
   Therefore, integrating term by term, we get
   $$
   \int_{\BB_n}|f(z)|^p\om(z)dV(z)\le N^{p-1}\sum_{k=1}^\infty |\lambda_k|^p,
   $$
   which combined with (\ref{equ1}) deduces the desired result.
   
    Finally, we only need to verify our claim, that is, $ \Lambda$ is bounded on $L^p_\om$. Let
  $$
 h(z)=(1-|z|^2)^{ -{t-n-1\over 2}},\,\,z\in \BB_n.
 $$
By Lemmas \ref{whatw} and \ref{whatwt}, we obtain
 \begin{align*}
 	\int_{\BB_n}{ (1-|w|^2)^{t-n-1} \over  |1-\langle z,w\rangle| ^t}h^2(w)\om(w)dV(w)
 	=&\int_{\BB_n}{ \om(w) \over |1-\langle z,w\rangle| ^t }dV(w)\\
 	\lesssim & {\hat{\om}(z) \over (1-|z|^2)^{ t-n}}\lesssim {1 \over (1-|z|^2)^{ t-n-1} }
 \end{align*}
 and
 \begin{align*}
 	\int_{\BB_n}{ (1-|w|^2)^{t-n-1} \over  |1-\langle z,w\rangle| ^t}h^2(z)\om(z)dV(z)
 	= & \int_{\BB_n}{(1-|w|^2)^{t-n-1} \om(z)\over (1-|z|^2)^{t-n-1}|1-\langle z,w\rangle| ^t }dV(z)\\
 	\lesssim &{\hat{\om}(w) \over (1-|w|^2)^{ t-n}}\lesssim {1 \over (1-|w|^2)^{ t-n-1} },
 \end{align*}
  which implies that  $ \Lambda$ is bounded on $L^p_\om$ by Schur's test. The proof is complete.
\end{proof}
  
  The following lemma extends the one dimensional result [\citen{lr}, Lemma 1] to higher dimensions and   classical weighted Bergman spaces [\citen{kw}, Lemma 2.2] to weighted Bergman spaces induced by doubling weights.

  \begin{lem}\label{fabD}	
  	Let $0<p\le q<\infty$, $\om\in \mathcal{D}$ and $0<r_2<r_1<1$. Then there is a constant $C:=C(\om,p,q, r_1,r_2)>0$ such that
  	$$
  	|f(a)-f(b)|^q\le C{ \rho(a,b)^q \over  \left(  (1-|a|)^{n}\hat{\om}(a) \right)^{q \over p } }\int_{D(a,r_1)}|f(z)|^p W(z) dV(z)
  	$$
  	for any $a\in\BB_n$, $b\in D(a,r_2)$ and $f\in A^p_\om$.
  \end{lem}
  
\begin{proof}
  Note that for any $z\in \BB_n \backslash \{0\}$ and $r>0$, $\Delta(z,r)$ is an ellipsoid consisting of all $w\in \BB_n$ such that 
\begin{equation}\label{deltaPQ}
  { |P_z(w)-c|^2  \over r^2s^2 }+{ |Q_z(w)|^2 \over r^2s }<1,
\end{equation}
  where
   \begin{align*}
   	c={ (1-r^2)z \over  1-r^2|z|^2 } {\rm ~and~} s={ 1-|z|^2 \over 1-r^2|z|^2 }.
\end{align*}
  This combined with (\ref{DPQ}) and (2.7) in \cite{kw} deduces that
  $$
  |f(0)-f(b)|^p\le C|b|^p\int_{\tanh(r_1)\BB_n} |f(z)|^pdV(z)
  $$
  for $b\in D(a,\tanh(r_2))$ and some $C:=C(r_1,r_2)>0$, where $a\in \BB_n$ and $0<r_2<r_1<1$.
  Write $R:=R(r_1)=\tanh r_1$.
  By Lemma \ref{lem1}, we obtain
   \begin{align*}
   	|f(a)-f(b)|^p = & |f\circ \varphi_a(0)-f\circ \varphi_a(\varphi_a(b))|^p\\
   	\le & C|\varphi_a(b)|^p\int_{R\BB_n} |f\circ \varphi_a(z)|^pdV(z)\\
   	\le & C \rho(a,b)^p \int_{D(a,R)} |f(z)|^p\left( { 1-|a|^2 \over |1-\langle z,a\rangle |^2 } \right) ^{n+1} dV(z)\\
   	\le & C_1  {\rho(a,b)^p \over (1-|a|)^{n}\hat{\om}(a) } \int_{D(a,R)} |f(z)|^p { \hat{\om}(z) \over 1-|z|}  dV(z)\\
   	= & C_1  {\rho(a,b)^p \over (1-|a|)^{n}\hat{\om}(a) } \int_{D(a,R)} |f(z)|^p W(z) dV(z)
   \end{align*}
   for some $C_1:=C_1(\om, r_1,r_2)>0$.
   
   Next, we consider the case $p<q$. Using the result of $p=q$ above and Lemma \ref{ApwW}, we get
    \begin{align*}
   	|f(a)-f(b)|^q\le & C^{q \over p} {\rho(a,b)^q \over \left( (1-|a|)^{n}\hat{\om}(a) \right) ^{q \over p}}\left( \int_{D(a,R)} |f(z)|^p  W(z)  dV(z) \right) ^{q \over p}\\
   	\le & C^{q \over p} {\rho(a,b)^q \|f\|^{q-p}_{A^p_W}\over \left( (1-|a|)^{n}\hat{\om}(a) \right) ^{q \over p}} \int_{D(a,R)} |f(z)|^p  W(z)  dV(z)\\
   	\le & C_2 {\rho(a,b)^q \over \left( (1-|a|)^{n}\hat{\om}(a) \right) ^{q \over p}} \int_{D(a,R)} |f(z)|^p  W(z)  dV(z)
   	  \end{align*}
   	  for some $C_2:=C_2(\om,p,q, r_1,r_2)>0$.
   	  The proof is complete.
\end{proof}

  \section{Carleson measure}
In this section, we characterize the (vanishing) $q$-Carleson measure for $A^p_\om$ for full $0<p,q<\infty$.
To present our results, we introduce a notation first.
  Given $\om \in \mathcal{D},0<r<1$, $0<s<\infty$ and a positive Borel measurev $\mu$ on $\BB_n$, the weighted mean function $\hat{\mu}_{\om,r,s}$ is defined by 
  $$
  \hat{\mu}_{\om,r,s}(z):={ \mu(D(z,r)) \over \om(D(z,r))^s },\,\,z\in\BB_n.
  $$
  Write $ \hat{\mu}_{\om,r}:=\hat{\mu}_{\om,r,1}$ for simplicity.

To  characterize the $q$-Carleson measure for $A^p_\om$ in the case $p\le q$, we need to prove the following lemma.

\begin{lem}\label{hatmu}
	Let $ 0<r<1$, $0<p\le q<\infty$, $\om\in \mathcal{D}$ and $\mu$ be a positive Borel measure on $\BB_n$. Then for any $f\in A^p_\om$, one has
	$$
	\int_{\BB_n}|f(z)|^qd\mu(z)\lesssim \|f\|^{q-p}_{A^p_\om}\int_{\BB_n}|f(z)|^p  \hat{\mu}_{\om,r,{q \over p}}(z)W(z)dV(z).
	$$
\end{lem}

\begin{proof}
		For $f\in A^p_\om$, using the subharmonic property of $ |f|^p$ (see Lemma 2.24 in \cite{z}), (\ref{wDzr}),  Lemmas \ref{lem1} and \ref{ApwW}, we have
	\begin{align*}
		|f(z)|^q\lesssim & \left( { 1 \over (1-|z|^2)^{n+1}} \int_{D(z,r)} |f(w)|^pdV(w) \right) ^{q \over p}\\
		\lesssim & \left( {1 \over \hat{\om}(z)(1-|z|^2)^n } \int_{D(z,r)} |f(w)|^p W(w) dV(w) \right) ^{q \over p}\\
		\lesssim &  {1 \over \left( \hat{\om}(z)(1-|z|^2)^n \right) ^{q\over p}} \|f\|^{q-p}_{A^p_\om}\int_{D(z,r)} |f(w)|^p W(w) dV(w).
	\end{align*}
	Combined with (\ref{DPQ}), (\ref{deltaPQ}), Proposition 1 in \cite{dl2} and Lemma \ref{lem1}, we see that 
 \begin{equation}\label{omDS}
	\om (D(a,r)) \approx \om(S_a)\approx \hat{\om}(a) (1-|a|^2)^n.
\end{equation}
	Therefore, by Fubini's theorem and Lemma \ref{ApwW},
	\begin{align*}
		\int_{\BB_n} |f(z)|^qd\mu(z) \lesssim & \|f\|^{q-p}_{A^p_\om}\int_{\BB_n} {1 \over \left( \hat{\om}(z)(1-|z|^2)^n \right) ^{q\over p}} \int_{D(z,r)} |f(w)|^p W(w) dV(w) d\mu(z)\\
		\lesssim & \|f\|^{q-p}_{A^p_\om}\int_{\BB_n} |f(w)|^p W(w) { \mu(D(w,r)) \over  \left( \hat{\om}(w)(1-|w|^2)^n \right) ^{q\over p} }dV(w)\\
		\lesssim & \|f\|^{q-p}_{A^p_\om}\int_{\BB_n}|f(w)|^p  \hat{\mu}_{\om,r,{q \over p}}(w)W(w)dV(w).
	\end{align*}
	The proof is complete.
\end{proof}

The following result gives the characterization of the (vanishing) $q$-Carleson measure in the case $A^p_\om$ for $p\le q$.

\begin{thm}\label{pqCarleson}
	Let $0<p\le q<\infty$, $\om\in \mathcal{D}$ and $\mu$ be a positive Borel measure on $\BB_n$. 
	Then there is $ r_0:=r_0(\om)\in (0,1)$ such that the following statements hold:
	\begin{enumerate}
		\item[(i)]
	 $ \mu$ is a $q$-Carleson measure for $A^p_\om$ if and only if for some (or all) $r\in [r_0,1)$ such that $\hat{\mu}_{\om,r,{q \over p}} \in L^\infty(\BB_n)$;
	 	\item[(ii)] $ \mu$ is a vanishing $q$-Carleson measure for $A^p_\om$ if and only if for some (or all) $r\in [r_0,1)$ such that $\hat{\mu}_{\om,r,{q \over p}}(a) \to 0$ as $|a|\to 1$.
	 \end{enumerate}
\end{thm}

\begin{proof}
	(i). If $\hat{\mu}_{\om,r,{q \over p}} \in L^\infty(\BB_n)$,
	then  $ \mu$ is a $q$-Carleson measure for $A^p_\om$ by Lemmas \ref{hatmu} and \ref{ApwW}.
 We next verify the necessity. Suppose that  $ \mu$ is a $q$-Carleson measure for $A^p_\om$. Consider the test function
  \begin{equation}\label{fa}
f_a(z)={1\over \hat{\omega}(a)^{1\over p} (1-|a|^2)^{n\over p}}\left(    {1-|a|^2\over |1-\langle z,a\rangle|}\right)  ^{ \gamma+n\over p} ,\,\,a,z\in \BB_n,
\end{equation}
where $\gamma$ is large enough.
From Lemma \ref{lem1} and Lemma 6 in \cite{dlls2}, $f_a\in A^p_\omega$.
Applying (\ref{wDzr}), (\ref{omDS}) and the fact $W(S_a)\approx \om(S_a),a\in \BB_n$ ( see the proof of Lemma \ref{ApwW}), we have
	\begin{equation}\label{wDhw}
	\begin{aligned}
	(1-|a|^2)^{n}\hat{\om}(a)\approx & {\hat{\om}(a) \over 1-|a|^2} \int_{D(a,r)}dV(z)\approx \int_{D(a,r)} {\hat{\om}(z) \over 1-|z|^2} dV(z)\\
	= & W(D(a,r))\approx W(S_a)\approx \om(S_a) \approx \om(D(a,r)).
\end{aligned}
\end{equation}
Therefore, for $a\in \BB_n$, by  the assumption,
$$
\hat{\mu}_{\om,r,{q \over p}}(a) \approx { \mu(D(a,r)) \over  \left( \hat{\om}(a)(1-|a|^2)^n \right) ^{q\over p} } \lesssim \|f_a\|^q_{L^q_\mu} \lesssim \|f\|^q_{A^p_\om}\approx 1,
$$
which means that
$$
\mu(D(a,r))\lesssim \left( \hat{\om}(a)(1-|a|^2)^n \right) ^{q\over p} \lesssim  (\om(D(a,r)))^{q\over p}.
$$
This completes the proof of (i).

(ii) Assume that $ \mu$ is a vanishing $q$-Carleson measure for $A^p_\om$.
Consider the test function $f_a$ defined as (\ref{fa}).
Then $f_a\in A^p_\omega$ and converges to $0$ uniformly on any compact subset of $\BB_n$ as $|a|\to 1$.
Therefore, by Lemma \ref{cpt},
$$
0=\lim_{|a|\to 1^-} \|f_a\|^q_{L^q_\mu}\ge \lim_{|a|\to 1^-} \int_{D(a,r)} |f_a(z)|^qd\mu(z) \gtrsim \lim_{|a|\to 1^-} { \mu(D(a,r)) \over  \left( \hat{\om}(a)(1-|a|^2)^n \right) ^{q\over p} },
$$
which means that $\hat{\mu}_{\om,r,{q \over p}}(a) \to 0$ as $|a|\to 1$.

Conversely, assume that $\hat{\mu}_{\om,r,{q \over p}}(a) \to 0$ as $|a|\to 1$. Then for any $\varepsilon>0$, there is $r:=r(\om) \in (0,1)$ such that ${ \mu(D(a,r)) \over  \left( \hat{\om}(a)(1-|a|^2)^n \right) ^{q\over p} }<\varepsilon$ when $|a|>r$.
Put $d\mu_r(z):=\chi_{r\le |z|<1}d\mu(z)$.
If $ |a|\ge r$, then $\mu_r(D(a,r))\le \mu(D(a,r))$. 
If $0<|a|<r$, by the fact $D(a,r) = \varphi_a (D(0,r))$, then
\begin{align*}
	\mu_r(D(a,r))  = & (1-|a|^2)^{n+1} \int_{D(0,r)\backslash D(0,\tanh^{-1}r)} { d\mu(z)  \over |1-\langle z, a \rangle |^{2(n+1)} } \\
	\lesssim & \mu(D(0,r)\backslash D(0,\tanh^{-1}r)).
	\end{align*}
	From [\citen{z}, page 64], we see that there is a finite sequence $\{a^{(1)},\cdots,a^{(N)}\}$ such that
	$$
	D(0,r)\backslash D(0,\tanh^{-1}r)\subset \cup_{j=1}^N D(a^{(j)},\delta) 
	$$
	 for some $\delta\in (0,1)$, where $a^{(j)}\in D(0,r)\backslash D(0,\tanh^{-1}r)$, $j=1,\cdots,n$.
	 Therefore,
	 $$
	 	\mu_r(D(a,r)) \lesssim \cup_{j=1}^ND(a^{(j)},\delta) \lesssim \varepsilon,
	 $$
	which implies that $\hat{\mu}_{\om,r,{q \over p}}(a) \lesssim \varepsilon $.
Thus, $\|I_d\|^q_{A^p_\om\to L^q_{\mu_r}} \lesssim \varepsilon$.
Let $\{f_k\}\in A^p_\om$ and converge to $0$ uniformly on any compact subset of $\BB_n$. Then 
\begin{align*}
	\lim\sup_{k\to \infty}\|f_k\|^q_{L^q_\mu} = &	\lim\sup_{k\to \infty} \left( \int_{\BB_n} |f_k(z)|^qd\mu_r(z) + \int_{r\BB_n}  |f_k(z)|^qd\mu(z)\right) \\
	\lesssim & \varepsilon 	\lim\sup_{k\to \infty}\|f_k\|^q_{A^p_\om}\lesssim \varepsilon,
	\end{align*}
	which gives that $\mu$ is a vanishing $q$-Carleson measure for $A^p_\om$ by the  arbitrariness of $\varepsilon$.
\end{proof}

Next, we consider the case $q>p$.

\begin{thm}\label{qpCarleson}
	Let $0<q<p<\infty$, $\om\in \mathcal{D}$ and $\mu$ be a positive Borel measure on $\BB_n$. 
	Then there is $ r_0:=r_0(\om)\in (0,1)$ such that the following statements are equivalent: 
 \begin{enumerate}
 	\item[(i)] $ \mu$ is a $q$-Carleson measure for $A^p_\om$;
 	\item[(ii)] $ \mu$ is a vanishing $q$-Carleson measure for $A^p_\om$;
 	\item[(iii)] For some (or all)  $r\in [r_0,1)$, $\hat{\mu}_{\om,r} \in L^{p\over p-q}_W$; 
 	\item[(iv)]  For some (or all)  $r\in [r_0,1)$ and any $\delta$-lattice $\{a_k\}$,
 	$$\hat{\mu}_{\om,r,{q \over p}}(a_k) \in l^{p\over p-q}.$$ 
 	\end{enumerate}
\end{thm}

\begin{proof}
    $(i) \Rightarrow (iv)$. Assume that $ \mu$ is a $q$-Carleson measure for $A^p_\om$. 
    Consider the function
    $$
    F_t(z)=\sum_{k=1}^\infty\lambda_k r_k(t)f_k(z),\,t\in[0,1],\,z\in \BB_n,
    $$
   where $\lambda=\{\lambda_k\}\in l^p$, $r_k$ are the Rademacher functions and 
   $$
   f_k(z)= {(1-|a_k|^2)^{t-{n\over p}}  \over  \hat{\om}(a_k)^{1\over p} (1-\langle z,a_k\rangle )^t},\,z\in \BB_n, 
   $$
   for some sufficiently large $t>0$. Then $F_t\in A^p_\om$ with $ \|F_t\|_{A^p_\om}\lesssim \|\lambda\|_{l^p}$ by Lemma \ref{Flambda}. 
   Thus, by the assumption, we obtain that
   \begin{align*}
   	\int_{\BB_n} \left| \sum_{k=1}^\infty\lambda_k r_k(t)f_k(z) \right| ^q d\mu(z)\le \|I_d\|^q_{A^p_\om\to L^q_\mu} \|F_t\|^q_{ A^p_\om } \lesssim \|I_d\|^q_{A^p_\om\to L^q_\mu}  \|\lambda \|^q_{l^q}.
   \end{align*}
   Then integrating with respect to $t$ on $[0,1]$, and employing Fubini's theorem and Khinchine's inequality, 
   $$
   \int_{\BB_n}\left( \sum_{k=1}^\infty |\lambda_k|^2 |f_k(z)|^2 \right) ^{q\over  2} d\mu(z) \lesssim \|I_d\|^q_{A^p_\om\to L^q_\mu} \|\lambda \|^q_{l^q}.
   $$
   By Lemma \ref{cover}, we get
    \begin{align*}
    \left( \sum_{k=1}^\infty |\lambda_k|^2 |f_k(z)|^2 \right) ^{q\over  2} &\gtrsim \sum_{k=1}^\infty |\lambda_k|^q |f_k(z)|^q \chi_{D(a_k,4r)}(z)\nonumber\\
    & \approx  
   \sum_{k=1}^\infty |\lambda_k|^q { \chi_{D(a_k,4r)}(z) \over \hat{\om}(a_k)^{q\over p} (1-|a_k|^2)^{{qn\over p}} }.
    	 \end{align*}
    Therefore,
    $$
    \sum_{k=1}^\infty |\lambda_k|^q { \mu(D(a_k,4r)) \over  \left( \hat{\om}(a_k) (1-|a_k|^2)^n \right) ^{q\over p} } \lesssim \|I_d\|^q_{A^p_\om\to L^q_\mu} \|\lambda \|^q_{l^q},
    $$
    which combined with the duality $$\left( l^{p\over q} \right)^\ast =l^{p\over p-q} $$ yields that (iv) holds.
    
     $(iv) \Rightarrow (iii)$. Suppose that $(iv)$ holds.
     By Lemma \ref{cover},
    we can choose $ \tilde{r}\in (0,1-r)$ and $N\ge 1$ such that
     $\BB_n \backslash \varepsilon\BB_n \subset \cup_{k\ge N} D(a_k,\tilde{r})$ for some $0<\varepsilon_0<\varepsilon <1 $.
     Note that $\hat{\mu}_{\om,r} $ is continuous on $\BB_n$.
     Therefore, by (\ref{wDhw}),
      \begin{align*}
    &  \int_{\BB_n} \hat{\mu}_{\om,r}(z) ^{ p \over p-q} W(z) dV(z) \\
     = & \left( \int_{\varepsilon \BB_n}  +\int_{\BB_n \backslash \varepsilon\BB_n} \right) \hat{\mu}_{\om,r}(z) ^{ p \over p-q} W(z) dV(z)\\
      \le & \int_{\varepsilon \BB_n} { \hat{\om} (z) \over 1-|z| } dV(z)+  \sum_{k \ge N} \int_{D(a_k,\tilde{r})} \hat{\mu}_{\om,r}(z) ^{ p \over p-q} { \hat{\om} (z) \over 1-|z| } dV(z) \\
      \lesssim & \hat{\om}(0) + \sum_{k \ge N} { \hat{\om} (a_k) \over 1-|a_k| } \int_{D(a_k,\tilde{r})} \left(  { \mu( D(z, r) ) \over \hat{\om}(z)	(1-|z|)^n } \right) ^{ p \over p-q}  dV(z)\\
       \lesssim & \hat{\om}(0) + \sum_{k \ge N} \left(  { \mu( D(a_k, r+\tilde{r} ) ) \over \left(  \hat{\om}(a_k)	(1-|a_k|)^n \right) ^{q \over p} } \right) ^{ p \over p-q} <\infty,
      \end{align*}
      which means that $(iii)$ holds.
      
        $(iii) \Rightarrow (ii)$. Suppose that $(iii)$ holds.
     Let   $ \{f_k \}$ be a bounded sequence in $A^p_\om $ and converge to $0$ uniformly on any compact subset of $\BB_n$. 
     By the assumption, for any $ \varepsilon >0$, there is $ \varepsilon_0 \in (0,1)$ such that 
     $$
     \int_{\BB_n \backslash \varepsilon_0 \BB_n} \hat{\mu}_{\om,r}(z) ^{ p \over p-q} W(z) dV(z)  <\varepsilon ^{ p \over p-q }.
     $$
     Therefore, using the subharmonic property of $|f_k|^q$, Fubini's theorem, H\"older's inequality and Lemma \ref{ApwW}, we have 
     \begin{align*}
  \| f_k\|^q_{L^q_\mu} = & \int_{\BB_n} |f_k(z)|^q d\mu(z) 
        \lesssim  \int_{\BB_n} { 1 \over (1-|z|)^{n+1} }\int_{D(z,r)} |f_k(w)|^q dV(w) d\mu(z)\\
     \lesssim & \int_{\BB_n} |f_k(w)|^q { \mu (D(w,r)) \over  (1-|w|)^{n+1} } dV(w)\\
     \lesssim &  \int_{\BB_n} |f_k(w)|^q \hat{\mu}_{\om,r}(w) W(w) dV(w)\\
     = & \left( \int_{\varepsilon_0 \BB_n}+\int_{\BB_n \backslash \varepsilon_0 \BB_n} \right) |f_k(w)|^q \hat{\mu}_{\om,r}(w) W(w) dV(w)\\
     \le & \left( \int_{\varepsilon_0 \BB_n} |f_k(w)|^q W(w) dV(w) \right) ^{ q \over p} \left( \int_{\BB_n} \hat{\mu}_{\om,r}(w)^{ p \over p-q} W(w) dV(w) \right) ^{ p-q \over p}\\
     &+ \left( \int_{ \BB_n} |f_k(w)|^q W(w) dV(w) \right) ^{ q \over p} \left( \int_{\BB_n \backslash \varepsilon_0 \BB_n} \hat{\mu}_{\om,r}(w)^{ p \over p-q} W(w) dV(w) \right) ^{ p-q \over p}\\
     \lesssim & \varepsilon \|f_k\|^q_{A^p_\om},
      \end{align*}
     which implies that $(ii)$ holds.
     
     $(ii) \Rightarrow (i)$. It is obvious. The proof is complete.
\end{proof}
	
\section{Main results}
In this section, we prove Theorems \ref{pqbddcom} and \ref{qpcom}. 
To present our results, we introduce some notations first.
The measures $\eta$ and $\sigma_r$ are defined by
$$
\eta=\eta_{\varphi,u}+ \eta_{\psi,u},
$$
$$
\sigma_r = \sigma_{\varphi,r} + \sigma_{\psi,r},
$$
where the measures $\eta_{\varphi,u}$ and $ \sigma_{\varphi,r}$ are defined by
$$
\eta_{\varphi,u} = \left(  |\rho u|^q d\mu \right) \circ \varphi^{-1},
$$
$$
\sigma_{\varphi,r} = \left( \chi_{G_r}|u-v|^q d\mu \right) \circ \varphi^{-1},
$$
and the measures $\eta_{\psi,u},  \sigma^\alpha_\psi$ and $ \sigma_{\psi,r}$ are defined similarly.

The following theorem, encompassing the content of Theorem \ref{pqbddcom}, is the main result of this subsection.

\subsection{The case $p\le q$}

We first investigate some sufficient conditions for $ uC_\varphi-vC_\psi: A^p_\om \to L^q_\mu$ to be bounded and compact.
	
\begin{thm}\label{pqs}
	Let $0<r<1$, $0<p\le q<\infty$ and $\om\in \mathcal{D}$, $\varphi,\psi \in \mathcal{S}(\BB_n)$,  $u,v$ be weights
	and $\mu$ be a positive Borel measure on $\BB_n$. 
	If $ \eta+\sigma_{\varphi,r}$ or $ \eta+\sigma_{\psi,r}$ is a (resp. vanishing) $q$-Carleson measure for $A^p_\om$, then $ uC_\varphi-vC_\psi: A^p_\om \to L^q_\mu$ is bounded (resp. compact).
\end{thm}

\begin{proof}
	By symmetry,
	we only need to consider the measure $\nu:= \eta+\sigma_{\psi,r}$. Suppose that $ \nu$ is a  $q$-Carleson measure for $A^p_\om$. Then Theorem \ref{pqCarleson} gives that $\hat{\nu}_{\om,\tilde{r}, {q \over p}} \in L^\infty(\BB_n)$ for some $\tilde{r} \in (r_0,1)$ with $r_0 =r_0(\om)>0$.
 For any function $f\in A^p_\om$, we have
	\begin{align*}
	\|(uC_\varphi-vC_\psi)f\|^q_{L^q_\mu} = & \left( \int_{G_r} + \int_{\BB_n \backslash G_r } \right)
		|u(z)f\circ \varphi(z)- v(z)f\circ \psi(z)|^qd\mu(z)\\
		=: & I+II.
	\end{align*}
	Using Lemma \ref{fabD}, Fubini's theorem, Lemma \ref{lem1} and (\ref{wDhw}),
	\begin{align*}
		A:= & \int_{G_r} |v(z)(f\circ \varphi(z)-f\circ \psi(z))|^q d\mu(z)\\
	\lesssim & \int_{G_r} { \rho(z)^q |v(z)|^q \over   \left(  (1-|\psi(z)|)^{n}\hat{\om}(\psi(z)) \right)^{q \over p } }\int_{D(\psi(z),\tilde{r})}|f(w)|^p W(w) dV(w) d\mu(z)\\
	\le & \int_{\BB_n} |f(w)|^p \int_{\psi^{-1}(D(w, \tilde{r}))} { \rho(z)^q |v(z)|^q \over   \left(  (1-|\psi(z)|)^{n}\hat{\om}(\psi(z)) \right)^{q \over p } } d\mu(z) W(w) dV(w)\\
		\lesssim & \int_{\BB_n} |f(w)|^p { \eta (D(w,\tilde{r})) \over \left(  (1-|w|)^{n}\hat{\om}(w) \right)^{q \over p } } W(w) dV(w)\\
	\lesssim & \int_{\BB_n} |f(w)|^p { \nu(D(w,\tilde{r})) \over \left(  (1-|w|)^{n}\hat{\om}(w) \right)^{q \over p } } W(w) dV(w)\\
	\lesssim &\int_{\BB_n} |f(w)|^p \hat{\nu}_{\om,\tilde{r}, {q\over p}}(w) W(w) dV(w).
		\end{align*}
Therefore, by Lemma \ref{hatmu}, we get
	\begin{align*}
	I \lesssim & \int_{G_r} \left( |(u(z)-v(z))f\circ \varphi(z)|^q +|v(z)(f\circ \varphi(z)-f\circ \psi(z))|^q  \right) d\mu(z)\\
	= & \int_{\BB_n} |f(z)|^q d\sigma_{\varphi,r}(z) + A\\
	\lesssim & \int_{\BB_n} |f(z)|^p \hat{\nu}_{\om,\tilde{r}, {q\over p}}(w) W(w) dV(w).
	\end{align*}
	For $II$, applying Lemma \ref{hatmu} again, we obtain
	\begin{align*}
		II \lesssim & \int_{\BB_n \backslash G_r } (  |u(z)f\circ \varphi(z)|^q + |v(z)f\circ \psi(z)|^q)d\mu(z)\\
		\le & {1 \over r^q} \int_{\BB_n \backslash G_r } \rho(z)(  |u(z)f\circ \varphi(z)|^q + |v(z)f\circ \psi(z)|^q)d\mu(z)\\
		\le & {1 \over r^q} \int_{ \BB_n} |f(z)|^q d\eta(z)\\
 \le & {1 \over r^q} \int_{ \BB_n} |f(z)|^q d\nu(z) \\
		\lesssim & \int_{\BB_n} |f(z)|^p \hat{\nu}_{\om,\tilde{r}, {q\over p}}(w) W(w) dV(w).
	\end{align*}
	Thus, by Lemma \ref{ApwW},
\begin{equation}\label{Tfnu}
		\|(uC_\varphi-vC_\psi)f\|^q_{L^q_\mu} 
		 \lesssim  \int_{\BB_n} |f(z)|^p \hat{\nu}_{\om,\tilde{r}, {q\over p}}(z) W(z) dV(z) 
		 \lesssim   \|\hat{\nu}_{\om,\tilde{r}, {q\over p}}\|_{L^\infty(\BB_n)}<\infty,
\end{equation}
	which means that $ uC_\varphi-vC_\psi: A^p_\om \to L^q_\mu$ is bounded.

Suppose that $ \nu$ is a vanishing $q$-Carleson measure for $A^p_\om$. 
Using Theorem \ref{pqCarleson}, for any $\varepsilon>0$, there is a constant $R\in(0,1)$ such that
$
\hat{\nu}_{\om,\tilde{r}, {q\over p}}(z)<\varepsilon
$
for $|z|>R$.
Let   $ \{f_k \}$ be a bounded sequence in $A^p_\om $ and converge to $0$ uniformly on any compact subset of $\BB_n$. Then by (\ref{Tfnu}) and Lemma \ref{ApwW}, we have
\begin{align*}
	&\lim\sup_{k\to \infty}	\|(uC_\varphi-vC_\psi)f_k\|^q_{L^q_\mu} \\
	\lesssim &	\lim\sup_{k\to \infty} \|f_k\|^{q-p}_{A^p_\om}\left( \int_{R\BB_n} +\int_{\BB_n \backslash R\BB_n } \right) |f_k(z)|^p \hat{\nu}_{\om,\tilde{r}, {q\over p}}(z) W(z) dV(z)\lesssim \varepsilon,
	\end{align*}
	which combined with Lemma \ref{cpt} gives that $ uC_\varphi-vC_\psi: A^p_\om \to L^q_\mu$ is compact.
\end{proof}

Before studying the necessary conditions, we introduce some notations.
For $\varphi,\psi\in \mathcal{S}(\BB_n)$ and $b\in \BB_n$, let
$$
Q_b(z):={1-\langle b,\varphi(z)\rangle  \over 1-\langle b,\psi(z)\rangle },z\in\BB_n.
$$
Given $r\in (0,1)$, $0<q,s<\infty$ and a positive Borel measure $\mu$ on $\BB_n$,  put
$$
R_{s,r,q}(a,b):=\int_{\varphi^{-1}(D(a,r))} |u-vQ_b^s| ^q d\mu
$$
for $a,b\in \BB_n$.
For $a\in \BB_n\backslash \{0\}$, set $ w_1(a):={a\over |a|}$ and extend it to any orthonormal basis $\{w_j(a)\}_{j=1}^n$ for $\CC^n$. For $N>0$ with $N^2(1-|a|)<1$, let
$$
a^{1,N}=\left( 1-N^2\sqrt{1-|a|}\right) w_1(a)
$$
and
$$
a^{j,N}=a^{1,N}+N\sqrt{1-|a|}w_j(a),\,j=2,\cdots,n.
$$
Since 
\begin{equation} \label{aN}
1-|a^{j,N}|^2= N^2(1-|a|)\times
\begin{cases}
	\left(2-N^2(1-|a|)\right) {~\rm if~} j=1\\
		\left(1-N^2(1-|a|)\right) {~\rm if~} j\ne1
	\end{cases},
\end{equation} 
we have $a^{j,N}\in\BB_n ,j=1,\cdots,n$.
In case $a=te_1 $, take $w_j(te_1):=e_j,j=1,\cdots,n$.
For $t\in (0,1)$, $a=te_1$ and $N\ge 1$ with $N^2(1-|a|)<1$,  set 
$$
J_N(a):= \left\lbrace a, a^{j,N^2}, a^{j,N^3} \right\rbrace _{j=1}^n.
$$
From Lemma 3.14 in \cite{cckp}, we see the following result hold.

\begin{lem}\label{muR}
	Let $0<q<\infty$,  $0<r,R<1$, $s>0$, $\varphi,\psi \in \mathcal{S}(\BB_n)$ and $u,v$ be weights.  Put $ \mu:=\eta_{\varphi,u} + \sigma_{\varphi,r}$.
	Then there exist constants $N:=N(s,R)\ge 1$, $t_0:=t_0(s,R,N)\in (0,1)$ and $C:=C(q,s,R,N)>0$ such that
	$$
	\mu(D(te_1,R))\le C \sum_{b\in J_N(te_1)} R_{s,r,q}(te_1,b)
	$$
	for all $t\in (t_0,1)$.
\end{lem}

Next, we give some necessary conditions for $ uC_\varphi-vC_\psi: A^p_\om \to L^q_\mu$ to be bounded and compact.

\begin{thm}\label{pqn}
	Let $0<p\le q<\infty$ and $\om\in \mathcal{D}$. Let $\varphi,\psi \in \mathcal{S}(\BB_n)$, $u,v$ be weights
	 and $\mu$ be a positive Borel measure on $\BB_n$. 
	If $ uC_\varphi-vC_\psi: A^p_\om \to L^q_\mu$ is bounded (resp. compact), then the measures $ \eta_{\varphi,u} +\sigma_{\varphi,r}$ and $ \eta_{\psi,u} +\sigma_{\psi,r}$ are (resp. vanishing)  $q$-Carleson measure for $A^p_\om$.
\end{thm}

\begin{proof} Write $\nu:= \eta_{\varphi,u} +\sigma_{\varphi,r}$.  Let $0<\delta<\infty$. Fix numbers $N:=N(s,R)$ and $t_0:=t_0(s,R,N)$ provided by Lemma 4.2,
and choosing $\delta_1=\delta_1(t_0)\in(0,1)$ such that
$$\bigcup_{z\in t_0\overline{\BB_n}}D(z,\delta)\subset \delta_1\overline{\BB_n}.$$
Since
\begin{equation} 
\begin{split}
		\rho|u|\le &
	\rho(|u|+|v|)
	  \leq \frac{ (|\varphi-\psi|(|u|+|v|)  }{ 1-\delta_1  } 
	  \\
	&\leq \frac{ 2|u\varphi-v\psi|+|u-v|(|\varphi|+|\psi|) }{ 1-\delta_1    }   \\
&\leq \frac{ 2|(uC_\varphi-vC_\psi)(id)|+2|(uC_\varphi-vC_\psi)(1)| }{ 1-\delta_1    },  \nonumber
	\end{split}
\end{equation}
so that 
\begin{equation}\nonumber
\nu(\delta_1\overline{\BB_n})\lesssim \|uC_\varphi-vC_\psi)(id)\|_{L^q_\mu}^q+\|uC_\varphi-vC_\psi)(1)\|_{L^q_\mu}^q.
\end{equation}
By  Lemma \ref{whatw},  there exists a constant $r_1>0$ such that
$$\inf_{z\in t_0\overline{\BB_n}}\hat\omega(z)>r_1.$$
This, together with (\ref{omDS}), yields
\begin{equation} \begin{split}\label{nuEQ}
	\sup_{z\in t_0\overline{\BB_n}}\hat{\nu}_{\om,R,{q\over p}}(z)
	& \leq \frac{\nu(\delta_1\overline{\BB_n})}{r_1^{q\over p}(1-t_0^2)^{qn\over p} } \\
 &\lesssim  \|uC_\varphi-vC_\psi)(id)\|_{L^q_\mu}^q+\|uC_\varphi-vC_\psi)(1)\|_{L^q_\mu}^q.
\end{split} \end{equation}

We now turn to the estimate  $\hat{\nu}_{\om,R,{q\over p}}$  for  $\BB_n \backslash t_0\overline{\BB_n}$. Consider the test function
	$$
	f_{b,s}={ K_b^s \over \|K_b^s\|_{A^p_\om} },b\in J_N(te_1),
	$$
	where $s>{ \lambda(\om)+n \over p }$ for some $\lambda(\om)>0$.
	Then by (\ref{Kas}) and (\ref{wDzr}),
	\begin{align*}
	\|(uC_\varphi-vC_\psi)f_{b,s}\|^q_{L^q_\mu} = & { 1 \over \|K_b^s\|_{A^p_\om}^q } \int_{\BB_n} \left| { u \over (1- \langle b,\varphi\rangle )^s} - { v \over (1- \langle b,\psi\rangle )^s }   \right|^q d\mu \\
	\gtrsim& { 1 \over (1-t)^{-qs+{qn \over p}}\hat{\om}(t)^{q\over p} } \int_{\varphi^{-1}(D(te_1,R))} { |u-vQ_b^s|^q \over |1- \langle b,\varphi\rangle|^{qs} } d\mu
	\\		
	\approx & { 1\over \left( (1-t)^{n}\hat{\om}(t)\right) ^{q\over p}} \int_{\varphi^{-1}(D(te_1,R))}  |u-vQ_b^s|^q  d\mu,
		\end{align*}
		for some $R\in(0,1)$.
	Therefore, using Lemma \ref{muR},
\begin{equation}\nonumber
	\begin{aligned}
		{ \nu(D(te_1,R)) \over \left( (1-t)^{n}\hat{\om}(t)\right) ^{q\over p}} \lesssim
	&	\sum_{b\in J_N(te_1)}	\|(uC_\varphi-vC_\psi)f_{b,s}\|^q_{L^q_\mu}
		\end{aligned}
	\end{equation}
	for all $t\in(t_0,1)$, 	which combined with Lemma \ref{lem1} and (\ref{aN}), we can find a constant $C>0$ such that
\begin{equation}\label{nue}
	\hat{\nu}_{\om,R,{q\over p}}(te_1)\lesssim { \nu(D(te_1,R)) \over \left( (1-t)^{n}\hat{\om}(t)\right) ^{q\over p}} 	\lesssim  \sup_{1-|a|\leq C(1-|t|)}\|(uC_\varphi-vC_\psi)f_{a,s}\|_{ L^q_\mu}^q 
	\end{equation}
	for all $t\in (t_0,1)$. For any $z\in\BB_n \backslash t_0\overline{\BB_n}$, we have $|z|\in(t_0,1)$.
Suppose $\mathcal{U}$ is the unitary operator on $\CC^n$, which maps $z$ to $|z|e_1$. Denote  $$\nu_{\mathcal{U}}:= \eta_{\mathcal{U}\circ\varphi,u} +\sigma_{\mathcal{U}\circ\varphi,r}.$$  Noting 
$$\widehat{\nu_\mathcal{U}}_{\om,R,{q\over p}}(z)=\hat{\nu}_{\om,R,{q\over p}}(\mathcal{U}^*z)$$
for all $z\in \BB_n$, and 
$$(uC_{\mathcal{U}\circ\varphi}-vC_{\mathcal{U}\circ\psi})f_{a,s}=(uC_\varphi-vC_\psi)f_{\mathcal{U}^*a,s}  $$
for all $a\in \BB_n$, where $\mathcal{U}^*=\mathcal{U}^{-1}.$ 
So, applying (\ref{nue}) with $\mathcal{U}\circ\varphi$ and $\mathcal{U}\circ\psi$ in place of $\varphi$ and $\psi$, we obtain
\begin{equation} \label{nuom}
	\hat{\nu}_{\om,R,{q\over p}}(z)=\widehat{\nu_\mathcal{U}}_{\om,R,{q\over p}}(|z|e_1)	\lesssim  \sup_{1-|a|\leq C(1-|z|)}\|(uC_\varphi-vC_\psi)f_{a,s}\|_{ L^q_\mu}^q.
	\end{equation}
	for all $z\in\BB_n \backslash t_0\overline{\BB_n}$. Using this and  (\ref{nuEQ}), 	we obtain that $ \hat{\nu}_{\om,R,{q\over p}}\in L^\infty(\BB_n)$. Then Theorem \ref{pqCarleson} yields that $\nu$ is a  $q$-Carleson measure for $A^p_\om$.
	
	If $ uC_\varphi-vC_\psi: A^p_\om \to L^q_\mu$ is  compact, then $\|(uC_\varphi-vC_\psi)f_{b,s}\|_{L^q_\mu} \to 0$ as $|b|\to 1$, where $s>{ \lambda(\om)+n \over p }$ for some $\lambda(\om)>0$. By Lemma \ref{lem1} and (\ref{nuom}), we have $\hat{\nu}_{\om,R,{q\over p}}(z) \to 0$ as $|z| \to 1$, which deduces that $\nu$ is vanishing $q$-Carleson measure for $A^p_\om$ by Theorem \ref{pqCarleson}.
	
	By symmetry,
	we  obtain that $ \eta_{\psi,v} +\sigma_{\psi,r}$ is also a (vanishing) $q$-Carleson measure for $A^p_\om$.
\end{proof}

From Theorems \ref{pqs} and \ref{pqn}, we obtain the following corollary, which contains Theorem \ref{pqbddcom}.

\begin{cor}\label{pqm}
	Let $0<p\le q<\infty$, $0<r<1$,  $\om\in \mathcal{D}$, $\varphi,\psi \in \mathcal{S}(\BB_n)$, $u,v$ be weights
	 and $\mu$ be a positive Borel measure on $\BB_n$. 
	Then the following statements are equivalent:
	\begin{enumerate}
		\item[(i)] $ uC_\varphi-vC_\psi: A^p_\om \to L^q_\mu$ is bounded (resp. compact);
		\item[(ii)] $\eta+ \sigma_{\varphi,r}$ or $\eta+ \sigma_{\psi,r}$ is (resp. vanishing) $q$-Carleson measure for $A^p_\om$;
		\item[(iii)] $\eta+ \sigma_{\varphi,r}$ and $\eta+ \sigma_{\psi,r}$ are (resp. vanishing) $q$-Carleson measure for $A^p_\om$.
	\end{enumerate}
\end{cor}

\subsection{The case $q<p$}

 We first explore the sufficient condition for $ uC_\varphi-vC_\psi: A^p_\om \to L^q_\mu$ to be compact.

\begin{thm}\label{qps}
	Let $0<q<p<\infty$, $0<r<1$,  $\om\in \mathcal{D}$, $\varphi,\psi \in \mathcal{S}(\BB_n)$, $u,v$ be weights
	and $\mu$ be a positive Borel measure on $\BB_n$. If $\eta+ \sigma_{\varphi,r}$ or $\eta+ \sigma_{\psi,r}$ is $q$-Carleson measure for $A^p_\om$, then $ uC_\varphi-vC_\psi: A^p_\om \to L^q_\mu$ is  compact.
\end{thm}

\begin{proof}
	Suppose that $\nu:=\eta+ \sigma_{\varphi,r}$  is $q$-Carleson measure for $A^p_\om$. By Theorem \ref{qpCarleson}, there is $r_0=r_0(\om)\in (0,1)$ such that $ \hat{\nu}_{\om,r} \in L^{p\over p-q}_W$ for some $r\in (r_0,1)$.
	For $f\in A^p_\om$, employing (\ref{Tfnu}) with $p=q$, H\"older's inequality and Lemma \ref{ApwW}, we have
	\begin{align*}
   \|(uC_\varphi-vC_\psi)f\|^q_{L^q_\mu} \lesssim & \int_{\BB_n} |f(z)|^q \hat{\nu}_{\om,r}(z) W(z) dV(z)\\
   \le & \left( \int_{ \BB_n} |f(z)|^pW(z)dV(z)\right) ^{q\over p} \left( \int_{ \BB_n} \hat{\nu}_{\om,r}(z)^{p\over p-q}W(z) dV(z) \right) ^{p-q\over p}\\
   \lesssim & \|f\|^q_{A^p_\om}\|\hat{\nu}_{\om,r}\|_{L_W^{p\over p-q}},
	\end{align*}
	which gives that $ uC_\varphi-vC_\psi: A^p_\om \to L^q_\mu$ is bounded.
	Let  $ \{f_k \}$ be a bounded sequence in $A^p_\om $ and converge to $0$ uniformly on any compact subset of $\BB_n$. 
	For any $\varepsilon>0$, there is $ R\in (0,1)$ such that
	$$
	\int_{\BB_n \backslash R\BB_n } \hat{\nu}_{\om,R}(z)^{p\over p-q}W(z) dV(z)<\varepsilon^{p \over p-q}.
	$$
Similar to the above,
	\begin{align*}
		&\lim\sup_{k\to\infty}\|(uC_\varphi-vC_\psi)f_k\|^q_{L^q_\mu} \\
\lesssim &	\lim\sup_{k\to\infty}\left(  \int_{\BB_n}+\int_{\BB_n \backslash R\BB_n } \right) |f_k(z)|^q \hat{\nu}_{\om,r}(z) W(z) dV(z)\\
		\le & \lim\sup_{k\to\infty} \left(\int_{R\BB_n} |f_k(z)|^pW(z)dV(z) \right) ^{q\over p} \|\hat{\nu}_{\om,R}\|_{L_W^{p\over p-q}}\\
		& +\sup_{k\ge 1} \|f_k\|^q_{A^p_\om}\left( \int_{\BB_n \backslash R\BB_n } \hat{\nu}_{\om,r}(z) W(z) dV(z) \right) ^{p-q \over p}\lesssim \varepsilon,
	\end{align*}
	which implies that $ uC_\varphi-vC_\psi: A^p_\om \to L^q_\mu$ is  compact by the arbitrariness of $\varepsilon$ and Lemma \ref{cpt}. Similarly, if $\eta+ \sigma_{\psi,r}$ is $q$-Carleson measure for $A^p_\om$, we can obtain the same result.
\end{proof}

The following decomposition lemma is crucial for studying the necessary condition  for  $ uC_\varphi-vC_\psi: A^p_\om \to L^q_\mu$ to be bounded.

\begin{lem}\label{decomlem}
	Let $0<\delta<1$, $N>0$, $r=r(N)\in \left( 0,\tanh^{-1}\sqrt{1-{ N^2 \over 2(N^2+1)^2 } }\right) $, $M=M(n,\delta,r,N)$ be a positive integer and $\{a_k\}$ be a $\delta$-separated sequence in $\BB_n$ with $ N^2(1-|a_k|)<{1\over 2}$ for all $k$. 
	If any collection of more than $M$ of the pseudohyperbolic balls $ \left\lbrace  D(a_k^{1,N}, r),\cdots, D(a_k^{n,N}, r) \right\rbrace_{k=1}^\infty $ contains no point in common, then $\{a_k\}$ is a union of $M+1$ separated sequences.
\end{lem}

\begin{proof}
	Write $a_{1,1}=a_1,a_{2,1}=a_2\cdots,a_{M+1,1}=a_{M+1}$.
	For $a_{L+2}$, by the assumption, there is $m\in \{1,2,\cdots, M+1\}$ such that $a_{M+2} \notin \left\lbrace  D(a_m^{1,N}, r),\cdots, D(a_m^{n,N}, r) \right\rbrace $, that is,
	$
	\beta(a_{M+2},a_m^{j,N})\ge r
	$
	for all $j\in \{1,2,\cdots, n\}$.
	Hence 	by the fact $\rho(z,w)=\tanh\beta(z,w)$ for all $z,w\in \BB_n$ and (4.5) in \cite{cckp},
	\begin{align*}
		\beta(a_{M+2},a_{m,1})= & \beta(a_{M+2}, a_m)\ge  \beta(a_{M+2},a_m^{j,N})-\beta(a_m,a_m^{j,N})\\
		\ge & r-\tanh^{-1}{ \sqrt{1-{ N^2 \over 2(N^2+1)^2 } }}:=R.
	\end{align*}
	Put $a_{m,2}=a_{M+2}$. Inductively, suppose that $L$ is a positive integer and $$\{a_k\}_{k=1}^L=\cup_{m=1}^{M+1}\{a_{m,k}\}_{k=1}^{N_m}$$ such that for any  $m\in  \{1,2,\cdots, M+1\}$ and $ 1\le i,j \le N_m$, $i\ne j$,
	$$
	\beta(a_{m,i},a_{m,j})\ge R.
	$$
	For $a_{L+1}$, we claim that there is $m_0 \in \{1,2,\cdots,M+1\}$ such that 
	$$
	a_{L+1} \notin \cup_{k=1}^{N_{m_0}}\left\lbrace  D(a_{m_0,k}^{1,N}, R),\cdots, D(a_{m_0,k}^{n,N},R ) \right\rbrace.
	$$
If not, suppose that $$a_{L+1}\in \cap_{m=1}^{M+1}\cup_{k=1}^{N_m}\left\lbrace  D(a_{m,k}^{1,N}, R),\cdots, D(a_{m,k}^{n,N}, R) \right\rbrace .$$  Then for any $m\in \{1,2,\cdots,M+1\}$, there is $k_m \in \{1,2,\cdots,N_m\}$ such that $$
a_{L+1}\in \left\lbrace  D(a_{m,k_m}^{1,N}, R),\cdots, D(a_{m,k_m}^{n,N}, R) \right\rbrace  ,
$$ 
Thus,
$$
	a_{L+1} \in \cap_{m=1}^{M+1}\left\lbrace  D(a_{m,k_m}^{1,N}, R),\cdots, D(a_{m,k_m}^{n,N}, R) \right\rbrace,
$$
which gives a contradiction, and  our claim is correct.
Then putting $ a_{m_0,N_{m_0+1}}=a_{L+1}$, we obtain
$$
\{a_k\}_{k=1}^\infty=\cup_{m=1}^{M+1}\{a_{m,k}\}_{k=1}^{\infty}
$$
and each $\{a_{m,k}\}_{k=1}^{\infty}$ is a $R$-separated sequence.
	\end{proof}
	
	 We next give some necessary conditions for $ uC_\varphi-vC_\psi: A^p_\om \to L^q_\mu$ to be bounded.

\begin{thm}\label{qpn}
	Let $0<q< p<\infty$ and $\om\in \mathcal{D}$, $\varphi,\psi \in \mathcal{S}(\BB_n)$, $u,v$ be weights
	and $\mu$ be a positive Borel measure on $\BB_n$. 
	If $ uC_\varphi-vC_\psi: A^p_\om \to L^q_\mu$ is bounded, then  $ \eta_{\varphi,u} +\sigma_{\varphi,r}$ and $ \eta_{\psi,v} +\sigma_{\psi,r}$ are  $q$-Carleson measure for $A^p_\om$.
\end{thm}

\begin{proof}
	For $N\ge 36$ and $\delta\in (0,1)$, let $ \{a_k\}$ be a $\delta$-separated sequence in $\BB_n$ such that $ N^4(1-|a_k|)<{1\over 8}$ for each $k$ and $\{b_k\}$ be a sequence given by one of $\{a_k\},\{a_k^{1,N}\},\cdots,$ $\{a_k^{n,N}\}$.
	Consider the function
 $$
F_t(z)=\sum_{k=1}^\infty\lambda_k r_k(t)f_k(z),\,t\in[0,1],\,z\in \BB_n,
$$
where $\lambda=\{\lambda_k\}\in l^p$, $r_k$ are the Rademacher functions and 
$$
f_k(z)= {(1-|a_k|^2)^{s-{n\over p}}  \over  \hat{\om}(a_k)^{1\over p} (1-\langle z,b_k\rangle )^s}
$$
for large $s$.  
Employing Lemma 4.4 in \cite{cckp} and Lemma \ref{decomlem}, there is a positive integer $M$ such that $\{b_k\}$ is a union of $M$ separated sequences. Thus, Lemma \ref{Flambda} yields that $F_t\in A^p_\om$ with $\|F_t\|_{A^p_\om}\lesssim\|\lambda\|_{l^p}$ for $t\in [0,1]$.
Since
 $ uC_\varphi-vC_\psi: A^p_\om \to L^q_\mu$ is bounded, we have
 \begin{align*}
 	\|(uC_\varphi-vC_\psi)F_t\|^q_{L^q_\mu}= & \int_{ \BB_n} \left| \sum_{k=1}^\infty \lambda_k r_k(t){(1-|a_k|^2)^{s-{n\over p}}  \over  \hat{\om}(a_k)^{1\over p} (1-\langle \varphi,b_k\rangle )^s} (u-vQ_{b_k}^s) \right| ^qd\mu\\
 	\lesssim & \|uC_\varphi-vC_\psi\|^q_{A^p_\om \to L^q_\mu} \|\lambda\|^q_{l^p}.
 	\end{align*}
Fix	$0<r<1$ with ${2 \over 1-\tanh r}\ge 36$.
 	Integrating with respect to $t$ on $[0,1]$, and
  applying Fubini's theorem, Khinchines's inequality and Lemma \ref{cover}, we get
   \begin{align*}
   	&\int_{ \BB_n} \sum_{k=1}^\infty |\lambda_k|^q\left|{(1-|a_k|^2)^{s-{n\over p}}  \over  \hat{\om}(a_k)^{1\over p} (1-\langle \varphi,b_k\rangle )^s} (u-vQ_{b_k}^s) \right| ^q \chi_{D(b_k,r)\circ\varphi} d\mu\\
   	\lesssim & \int_{ \BB_n} \left( \sum_{k=1}^\infty \left| \lambda_k {(1-|a_k|^2)^{s-{n\over p}}  \over  \hat{\om}(a_k)^{1\over p} (1-\langle \varphi,b_k\rangle )^s} (u-vQ_{b_k}^s) \right| ^2 \right) ^{q\over 2} d\mu.
   \end{align*}
   Since $ N\ge 36\ge{ 2\over 1-\tanh r}$, by (\ref{De1r}), $D(b_k,r) \subset S(w_1(b_k),N(1-|b_k|))$ for all $k$. Using Lemma 3.6(a) in \cite{cckp},
   $$
   {\chi_{D(b_k,r) \circ\varphi} \over |1-\langle \varphi,b_k\rangle|} \approx { \chi_{D(b_k,r) \circ\varphi} \over N^2(1-|b_k|) } \approx { \chi_{D(b_k,r) \circ\varphi} \over N^2(1-|a_k|) }
   $$
   for all $k$. Therefore,
    \begin{align*}
    	\sum_{k=1}^\infty {  |\lambda_k|^q R_{s,r,q}(a_k,b_k) \over \left(  (1-|a_k|)^{n}\hat{\om}(a_k)\right) ^{q\over p}} 
    	\lesssim & \int_{ \BB_n} \left( \sum_{k=1}^\infty \left| \lambda_k {(1-|a_k|^2)^{s-{n\over p}}  \over  \hat{\om}(a_k)^{1\over p} (1-\langle z,b_k\rangle )^s} (u-vQ_{b_k}^s) \right| ^2 \right) ^{q\over 2} d\mu\\
    	\lesssim & \|uC_\varphi-vC_\psi\|^q_{A^p_\om \to L^q_\mu} \|\lambda\|^q_{l^p},
    \end{align*}
     which combined with the duality $\left( l^{p\over q} \right)^\ast =l^{p\over p-q} $ deduces that
     $$
     \left\| \left\lbrace {  R_{s,r,q}(a_k,b_k) \over \left(  (1-|a_k|)^{n}\hat{\om}(a_k)\right) ^{q\over p}} \right\rbrace \right\| _{l^{ p\over p-q}}
     \lesssim  \|uC_\varphi-vC_\psi\|^q_{A^p_\om \to L^q_\mu}.
     $$
     Write $\nu:= \eta_{\varphi,u} +\sigma_{\varphi,r}$.
    Then by Lemmas \ref{lem1} and \ref{muR},
     \begin{align*}
     	 \hat{\nu}_{\om,r,{q \over p}}(a_k)  \lesssim
     		\sum_{b\in J_N(a_k)}{  R_{s,r,q}(a_k,b_k) \over \left(  (1-|a_k|)^{n}\hat{\om}(a_k)\right) ^{q\over p}}
     \end{align*}
 	 for all large $k$, since $|a_k|\to 1$.
     If $a_k\in t_0\overline{\BB}_n$ for some $t_0\in (0,1)$, then using (\ref{nuEQ}),
     $$
     \sup_{a_k\in t_0\overline{\BB}_n} \hat{\nu}_{\om,r,{q \over p}}(a_k)\lesssim
     \|uC_\varphi-vC_\psi)(id)\|_{L^q_\mu}^q+\|uC_\varphi-vC_\psi)(1)\|_{L^q_\mu}^q<\infty.
     $$
         Therefore, $$ \{\hat{\nu}_{\om,r,{q \over p}}(a_k)\}_{k=1}^\infty\in l^{ p\over p-q},$$
          which implies that $\nu$ is a  $q$-Carleson measure for $A^p_\om$ by Theorem \ref{qpCarleson}. By symmetry,
    we  obtain that $ \eta_{\psi,v} +\sigma_{\psi,r}$ is also a  $q$-Carleson measure for $A^p_\om$.
\end{proof}

From Theorems \ref{qps} and \ref{qpn}, we obtain the following corollary, which contains Theorem \ref{qpcom}.

 \begin{cor}
 	Let $0<q<p<\infty$, $0<r<1$,  $\om\in \mathcal{D}$, $\varphi,\psi \in \mathcal{S}(\BB_n)$, $u,v$ be weights 
	 and $\mu$ be a positive Borel measure on $\BB_n$. 
 	Then the following statements are equivalent:
 	\begin{enumerate}
 		\item[(i)] $ uC_\varphi-vC_\psi: A^p_\om \to L^q_\mu$ is bounded;
 		\item[(ii)] $ uC_\varphi-vC_\psi: A^p_\om \to L^q_\mu$ is  compact;
 		\item[(iii)] $\eta+ \sigma_{\varphi,r}$ or $\eta+ \sigma_{\psi,r}$ are  $q$-Carleson measure for $A^p_\om$;
 		\item[(iv)] $\eta+ \sigma_{\varphi,r}$ and $\eta+ \sigma_{\psi,r}$ are  $q$-Carleson measure for $A^p_\om$.
 	\end{enumerate}
 \end{cor}
 
 	{\bf Data Availability}  No data was used to support this study.\msk
	
	{\bf Conflicts of Interest}  The authors  declare that they have no conflicts of interest.\msk
	

	{\bf Acknowledgements} 	The corresponding author was supported by  NNSF of China (No. 12371131), STU Scientific Research Initiation Grant(No. NTF23004), 
LKSF STU-GTIIT Joint-research Grant (No. 2024 LKSFG06),  Guangdong Basic and Applied Basic Research Foundation (No. 2023A1515010614).\\

\end{document}